\documentclass[a4paper]{amsart}
\usepackage[active]{srcltx}
\usepackage[all]{xy}
\usepackage{subfig}
\usepackage{hyperref}
\usepackage{mathrsfs}

\usepackage{amsmath}
\usepackage{amssymb,latexsym}
\usepackage{mathrsfs}
\usepackage{graphics}
\usepackage{latexsym}
\usepackage{psfrag}
\usepackage{import}
\usepackage{verbatim}
\usepackage{graphicx}
\usepackage[usenames]{color}
\usepackage{pifont,marvosym}

\theoremstyle{plain}
\newtheorem{lemma}{Lemma}[section]
\newtheorem{theorem}[lemma]{Theorem}

\theoremstyle{definition}

\numberwithin{equation}{section}

\newcommand{\ve}{\varepsilon}

\newcommand{\erre}{\mathbb{R}}
\newcommand{\enne}{\mathbb{N}}

\begin{document}
\title[A note on a Residual subset of Lipschitz functions on metric spaces]{A note on a Residual subset of \\ Lipschitz functions on metric spaces}

\author{Fabio Cavalletti}

\address{RWTH, Department of Mathematics, Templergraben  64, D-52062 Aachen (Germany)}
\email{cavalletti@instmath.rwth-aachen.de}

\begin{abstract}
Let $(X,d)$ be a quasi-convex, complete and separable metric space with reference probability measure $m$.
We prove that the set of of real valued Lipschitz function with non zero point-wise Lipschitz constant $m$-almost everywhere
is residual, and hence dense, in the Banach space of Lipschitz and bounded functions.
The result is the metric analogous of a result proved 
for real valued Lipschitz maps defined on $\mathbb{R}^{2}$ by Alberti, Bianchini and Crippa in \cite{abc:sard}. 
\end{abstract}

\maketitle



\bibliographystyle{plain}

\section{Introduction}

In the context of metric spaces, say $(X,d)$, it is possible to look at the point-wise variation of a real valued map considering
\begin{equation}\label{E:Lipoint}
\textrm{Lip}\, f(x) : = \limsup_{y\to x, y \neq x} \frac{|f(x) - f(y)|}{ d(x,y)},
\end{equation}
that is called \emph{point-wise Lipschitz constant}. 
In the smooth framework $\textrm{Lip} \, f$ corresponds to the modulus of $\nabla f$: if $(X,d)$ is an 
open subset of $\erre^{d}$ endowed with the euclidean norm and $f$ is locally Lipschitz, then 
$\textrm{Lip}\, f  = |\nabla f|$ almost everywhere with respect to the Lebesgue measure.
Or more in general if $(X,d,m)$ is a metric measure space admitting a differentiable structure in the sense of Cheeger, 
see \cite{cheeger:lip}, \cite{kleinermac:measurelip} for the definitions, and $f$ is Lipschitz, then $\textrm{Lip}\, f  = | d f|$ $m$-a.e. 
where $df$ is the Cheeger's differential of $f$.

Once a point-wise information is given we are interested at looking at those points where the ``differential'' vanishes:
define the singular set of $f$ as follows
\[
S(f) : = \{ x \in X : \textrm{Lip}\, f(x) = 0\}.
\]

The classical Sard's Theorem states that if $f : \erre^{n} \to \erre$ is sufficiently smooth then the Lebesgue measure of $f(S(f))$ is 0.
As soon as the regularity assumption on $f$ is dropped, the conclusion of Sard's Theorem does not hold anymore and one may look for weaker property to hold.

%

The question is if it is possible to approximate any Lipschitz function with functions having 
negligible $S(f)$ with respect to a given reference measure. 

For real valued Lipschitz functions defined on $\erre^{2}$, with Lebesgue measure playing the role of the reference measure, a positive answer is contained in \cite{abc:sard}, see Proposition 4.10.
We prove the following.

\begin{theorem}\label{T:main}
Assume $(X,d)$ is a quasi-convex, complete and separable metric space and let $m$ be a Borel probability measure over it. 
The set of those $f \in D^{\infty}(X)$ so that $m(S(f)) =0$ is residual, and therefore dense, in $D^{\infty}(X)$. 
\end{theorem}

The Banach space $D^{\infty}(X)$ will be the space of bounded functions with bounded point-wise Lipschitz constant, 
endowed with the uniform norm. See below for a precise definition.
Recall that a set in a topological space is residual if it contains a countable intersection of open dense set. 
By Baire Theorem, a residual set in a complete metric space is dense.

\section{Setting}

Let $(X,d)$ be a metric space and $m$ is a Borel probability measure over $X$ so that $X$ coincides with its support.
For $f : X \to \erre$ the \emph{Lipschitz constant} of f is defined as usual by 
\[
\textrm{LIP}(f) : = \sup_{ x,y \in X,   x\neq y} \frac{|f(x)-f(y)|}{d(x,y)},
\]
and we say that $f$ is Lipschitz if $\textrm{LIP}(f)$ is a finite number.
Accordingly denote by $\textrm{LIP}^{\infty}(X)$ the space of bounded Lipschitz functions. 
The natural norm on $\textrm{LIP}^{\infty}(X)$ is given by 
\[
\| f\|_{\textrm{LIP}^{\infty}(X)} = \| f \|_{\infty} + \textrm{LIP}(f),
\]
where $\| \cdot \|_{\infty}$ is the uniform norm. The space of bounded Lipschitz functions endowed with 
$\| f\|_{\textrm{LIP}^{\infty}(X)}$ turns out to be a Banach space.
The point-wise version of $\textrm{LIP}(f)$ is given by the point-wise Lipschitz constant as defined in \ref{E:Lipoint}.
The corresponding space of bounded functions with bounded point-wise Lipschitz constant can be considered:
\[
D^{\infty}(X)  : = \{ f : X \to \erre : \| f \|_{\infty} + \| \textrm{Lip}\, f \|_{\infty} < \infty \}.
\]
A study of $D^{\infty}(X)$ and $\textrm{LIP}^{\infty}(X)$ can be found 
in \cite{durand:pointwiselip}. The following results are taken from \cite{durand:pointwiselip}.

It is straightforward to note that $\textrm{LIP}^{\infty}(X) \subset D^{\infty}(X)$ and 
for a general metric space this is the only valid inclusion. 
Examples of metric spaces and functions in $D^{\infty}(X)$ not satisfying a global Lipschitz bound can be constructed, 
see \cite{durand:pointwiselip}.
If $(X,d)$ is quasi-convex also the other inclusion holds and $\textrm{LIP}^{\infty}(X) = D^{\infty}(X)$ and 
the two semi-norms are comparable: there exists $C \geq 1$ so that 
\[
\| \textrm{Lip}\, f \|_{\infty } \leq \textrm{LIP}(f)  \leq C \| \textrm{Lip}\, f \|_{\infty }.
\]
Hence $D^{\infty}(X)$, or equivalently $\textrm{LIP}^{\infty}(X)$, 
endowed with the norm $\| \cdot \|_{\infty} + \| \textrm{Lip}\, (\cdot) \|_{\infty }$ is a Banach space. We will denote this norm with $\| \cdot\|_{D^{\infty}}(X)$.

Recall that a metric space $(X,d)$ is quasi-convex if there exists a constant $C\geq1$ such that for each pair of points $x,y \in X$
there exists a curve $\gamma$ connecting the two points such that  $l(\gamma) \leq C d(x,y)$, 
where $l(\gamma)$ denotes the length of $\gamma$ defined with the usual ``affine'' approximation: 
for $\gamma : [a,b] \to X$ its length $l(\gamma)$ is defined by
\[
l(\gamma) : = \sup \left\{ \sum_{i =1 }^{n} d(x_{i},x_{i+1}) : a = x_{1} < x_{2} < \dots < x_{n+1} = b, n\in \enne \right\}.
\]
Associated to the length $l(\gamma)$ there is the distance obtained minimizing it: 
\[
d_{L}(x,y) = \inf \{ l(\gamma) : \gamma_{0} = x, \gamma_{1} = y \}. 
\]
The function $d_{L}$ is indeed a distance on each component of accessibility by rectifiable paths, i. e. those paths having finite 
$l$. By quasi-convexity it follows that 
\[
d(x,y) \leq d_{L}(x,y) \leq C d(x,y), 
\]
with $C>1$. Hence $(X,d_{L})$ is a complete and separable metric space that is also a length space.
Clearly $(X,d_{L})$ has the same open sets of $(X,d)$. For a more detailed discussion on length spaces see \cite{burago}.

We will use the following notation. 
For $r>0$ and $z \in X$, we will denote with $B_{r}(z)$ the ball of radius $r$ centered in $z$. 
The complement in $X$ of a set $A$ will be denoted by $A^{c}$ and $\partial A$ denotes the topological boundary of $A$. 
The closure of $A$ is $cl(A)$ and the interior part $int(A)$.
Associated to a set we can consider the distance from it:
for $x \in X$ and $A \subset X$
\[
d(x,A) : = \inf_{w \in A} d(x,w).
\]

\section{The Result}

\begin{lemma}\label{L:measurable}
For any Borel function $f : X \to \erre$, the function $\emph{Lip}\, f : X \to \bar \erre$ is universally measurable.
\end{lemma}

\begin{proof}
In order to prove the claim we just have to show that the set $\{ x\in X :  \textrm{Lip}\, f (x) \geq a\}$ is Souslin for any $a \in \erre$.
Since $f$ is a Borel map then
\[
\bigcap_{n \in \enne } \left\{ (x,y) \in X \times X :  0< d(x,y) \leq \frac{1}{n}, \  \frac{|f(x) - f(y)|}{ d(x,y)} \geq a\right\}
\]
is a Borel set. Note that 
\[
\{ x\in X :  \textrm{Lip}\, f (x)\geq a\}  = 
P_{1} \left( \bigcap_{n \in \enne } \left\{ (x,y) \in X \times X :  0< d(x,y) \leq \frac{1}{n}, \  \frac{|f(x) - f(y)|}{ d(x,y)} \geq a\right\} \right),
\]
where $P_{1} : X \times X \to X$ denotes the projection on the first element. It follows from the definition of Souslin set 
that $\{ x\in X :  \textrm{Lip}\, f (x)\geq a\}$ is Souslin and the claim follows. 
\end{proof}

Then after Lemma \ref{L:measurable} it makes sense to look at those functions $f$ so that $m(S(f))=0$.
We will need the following

\begin{lemma}\label{L:auxiliary}
Let $K \subset X$ be a closed set and consider the length distance function from $K$ that is $g(x) : = d_{L}(x,K)$.
Then
\[
1 \leq \emph{Lip}\, g(x) \leq C, \qquad  \textrm{for } x \in K^{c},
\]
\end{lemma}

\begin{proof} 
{\it Step 1.} Assume that $d = d_{L}$ so that $(X,d)$ is also a length space and $g = d(x,K)$. 
Then fix $x \in K^{c}$: for any $z \in K$ and $y \in K^{c}$  it holds
\[
d(x,z) - d(y,z) \leq d(x,y)
\]
hence trivially $\textrm{Lip}\, g(x) \leq 1$. 

Consider now a minimizing sequence $z_{n} \in K$ for $x$, that is  that 
$g(x) \geq d(x,z_{n}) - 1/n$. From the length structure it follows that for any $n$ there exists $\gamma^{n} : [0,1] \to X$ rectifiable 
curve starting in $x$ and arriving in $z_{n}$ so that $d(x,z_{n}) \geq l(\gamma^{n}) - 1/n$. So for any $y_{n}$ in the image of $\gamma^{n}$
\[
\frac{ g(x) - g(y_{n})}{d(x,y_{n})} \geq \frac{ l(\gamma_{n}) - d(y_{n},z_{n}) - 2/n }{d(x,y_{n})}.
\]
Since $l(\gamma^{n}) \geq d(x,y_{n}) + d(y_{n},z_{n})$ it follows that 
\[
\frac{ g(x) - g(y_{n})}{d(x,y_{n})} \geq \frac{  d(x,y_{n}) - 2/n  }{d(x,y_{n})}.
\]
Since the only constrain on $y_{n}$ was to belong to the image of $\gamma^{n}$, we can choose $y_{n}$ so that the previous ratio 
converges to $1$. Hence $\textrm{Lip} \, g (x) = 1$.

{\it Step 2.} We now drop the assumption on the length structure of the space. Let $(X,d)$ be quasi-convex and $g(x) = d_{L}(x,K)$. Since $(X,d_{L})$ is a length space for any $x \in K^{c}$
\[
\limsup_{y\to x, y \neq x} \frac{|g(x) - g(y)|}{d_{L}(x,y)} = 1.
\]
Having $(X,d_{L})$ and $(X,d)$ the same open set, $K^{c}$ does not depend on the metric. Since $d \leq d_{L} \leq C d$
the claim follows.
\end{proof}

We can now prove Theorem \ref{T:main}. The proof uses now the ideas contained in Proposition 4.10 in \cite{abc:sard}.

\begin{theorem}
Assume $(X,d)$ is a quasi-convex, complete and separable space and let $m$ be a Borel probability measure over it. 
The set of those $f \in D^{\infty}(X)$ so that $m(S(f)) =0$ is residual in $D^{\infty}(X)$ 
and therefore dense.
\end{theorem}

\begin{proof} 
Consider the following sets 
\[
G : = \{ f \in D^{\infty}(X) :  m(S(f)) = 0\}, \qquad   G_{r} : = \{ f \in D^{\infty}(X) :  m(S(f)) < r \}.
\]
The claim is then to prove that $G$ is a residual set. Since $G = \cap G_{r}$, where the intersection runs over a sequence of $r$ converging to 0, the claim is proved once 
it is proved that each $G_{r}$ is open and dense in $D^{\infty}(X)$.

{\it Step 1.} The set $G_{r}$ is open in $D^{\infty}(X)$. Fix $f \in G_{r}$. Then there exists $\delta > 0$ so that 
\[
m\left(\{x\in X : \textrm{Lip}\, f(x) \leq \delta  \}\right) < r.
\]
Since for any $g \in D^{\infty}(X)$ it holds that 
\[
\textrm{Lip}\, f (x)  \leq  \textrm{Lip}\, g (x) + \textrm{Lip}\, (f -g)(x), 
\]
for any $g \in D^{\infty}(X)$ so that $\| g - f \|_{D^{\infty}}(X) \leq \delta$ it holds that 
\[
S(g) \subset \{x \in X : \textrm{Lip}\, f(x) \leq \delta \},
\]
and therefore $m(S(g)) < r$ and consequently $g \in G_{r}$.

{\it Step 2.} The set $G_{r}$ is dense in $D^{\infty}(X)$. Given $f \in D^{\infty}(X)$ and $\delta > 0$ we have to find $g \in G_{r}$ so that 
$\| f - g \|_{D^{\infty}(X)} \leq \delta$. Without loss of generality we can assume $m(S(f)) \geq r$.

For every $\ve > 0$ denote with $S(f)^{\ve}$ the $\ve$-neighborhood of the set of singular points of $f$, i.e. 
\[
S(f)^{\ve} = \{ z\in X : d(z, S(f)) < \ve \}.
\] 
The set $S(f)^{\ve}$ is open and denote by $K$ its complementary in $X$. Associated to $K$ we consider the distance function $\hat g$ as defined in 
Lemma \ref{L:auxiliary} that is $\hat g(x) : = d_{L}(x,K)$. 
A rough bound on $\hat g(x)$ can be given in terms of the ``diameter'' of $S(f)$:
\[
\hat g(x) \leq C \sup\{ d(x,z) : cl(S(f)^{\ve}) \},
\]
where $cl(S(f)^{\ve})$ stands for the closure of $S(f)^{\ve}$. Since to approximate with functions in $G_{r}$ we can 
make an error in measure strictly less than $r$ and since $m$ is a probability measure, we can assume $S(f)$ to have finite diameter and by inner regularity we can even assume it to be closed.
Therefore 
\[
\|\hat  g \|_{\infty} \leq M, \qquad M > 0.
\]
From Lemma \ref{L:auxiliary} we have $\textrm{Lip}\, \hat g (x) > 0$ for $x \in S(f)^{\ve}$ and  clearly $\textrm{Lip}\, \hat g (x) = 0$ for $x \in int (K)$, 
where $int(K)$ stands for the interior part of $K$.

Note that the boundary of $S(f)^{\ve}$ is contained in the set $\{ z : d(z,S(f)) = \ve \}$. 
Indeed $z \in \partial S(f)^{\ve}$ if and only if $d(z,S(f)) \geq \ve$ and for every $\eta > 0$ there exists a point $w \in X$ so that 
\[
d(z,w) \leq \eta, \qquad d(w,S(f)) < \ve.
\]
Let $\eta_{n}$ be a sequence converging to $0$ and $w_{n}$ the corresponding sequence converging to $z$. To each $w_{n}$ associate $x_{n} \in S(f)$ so that 
$d(w_{n},x_{n}) < \ve$. Then 
\[
d(z,x_{n}) \leq d(z,w_{n}) + d(w_{n},x_{n}) < \eta_{n} + \ve.
\]
Passing to the limit $d(z, S(f)) \leq \ve$ and therefore necessarily $d(z,S(f)) = \ve$.

Moreover for $\ve \neq \ve'$
\[
\{ z : d(z,S(f)) = \ve \} \cap \{ z : d(z,S(f)) = \ve' \} = \emptyset,
\]
hence there exists at most countably many $\ve$ so that $m(\{ z : d(z,S(f)) = \ve \}) > 0$.
Hence for any $r>0$ there exists $\ve>0$ so that 
\[
m(\{ z : d(z,S(f)) = \ve \}) = 0, \qquad  m(S(f)^{\ve} \setminus S(f)) < r,
\]
where the second expression holds because $S(f)$ is closed.
From what said so far, denoting $g : = f + (\delta/2M)\hat g$ is such that
\[
\| f - g \|_{D^{\infty}(X)} \leq \delta.
\]
To conclude the proof observe that $S(g) \subset S(f)^{\ve} \setminus S(f)$, hence by construction $g \in G_{r}$.
\end{proof}

\end{document}